\documentclass[pdftex]{amsart}
\usepackage{amsmath,amssymb,bbm,color}
\usepackage[english]{babel}
\usepackage[latin1]{inputenc}
\usepackage{float} 
\usepackage[norelsize, english, algosection, algoruled, noline]{algorithm2e}
\usepackage{latexsym}
\usepackage{amsmath}
\usepackage{amssymb}
\usepackage{graphicx}
\usepackage{amsfonts}

\usepackage[T1]{fontenc}

\newtheorem{thm}{Theorem}[section]

\newtheorem{definition}[thm]{Definition}

\newtheorem{rem}[thm]{Remark}

\newtheorem{lem}[thm]{Lemma}
\newtheorem{defn}[thm]{Definition}

\newcommand{\supp}{\mathrm{supp}}

\newcommand{\XX}{\mathcal{M}}
\newcommand{\Cc}{\mathcal{C}}
\newcommand{\Ac}{\mathcal{A}}

\newcommand{\Sc}{\mathcal{S}}

\newcommand{\Fc}{\mathcal{F}}

\newcommand{\Mon}{\mathcal{M}}
\newcommand{\Span}[1]{\<{#1}\>}

\def\<{\langle}
\def\>{\rangle}

\def\NN{\mathbb{N}}
\def\ZZ{\mathbb{Z}}
\def\kk{\mathbb{K}}
\def\Mon{\mathcal{M}}

\def\com{\mathcal{C}}
\def\kx{\mathbb{K}[x_1^{\pm},\ldots,x_n^{\pm}]}

\def\Fc{\mathcal{F}}
  
\def\Xb{\mathbf{X}}  
 
\def\xb{\mathbf{x}} 
\def\<{\langle}
\def\>{\rangle}

\newcommand{\rank}{\mathrm{rank}\,}
\newcommand{\vspan}[1]{\langle{#1}\rangle}

\newcommand{\setdeg}[2]{{#1}_{[{#2}]}}
\newcommand{\spandeg}[2]{\langle{#1}_{[{#2}]}\rangle}

\title{Toric Border Bases}

\author[B. Mourrain \& Ph. Tr\'ebuchet]{
{Bernard Mourrain}\\
{Inria, \'Equipe GALAAD}\\
{BP 93, 06902 Sophia Antipolis, France}\\
{Bernard.Mourrain@inria.fr} \\
\ \\
{Philippe Tr\'ebuchet}\\
{Sorbonne Universit\'es, UPMC, Equipe APR. LIP6, UMR 7606,}\\
{Inria, \'Equipe OURAGAN}\\
{4 place Jussieu}
{75256 Paris Cedex}
{Philippe.Trebuchet@lip6.fr}
}

\begin{document}
\maketitle

\begin{abstract}
  We extend the theory and the algorithms of Border Bases to systems
  of Laurent polynomial equations, defining ``toric'' roots. Instead
  of introducing new variables and new relations to saturate by the
  variable inverses, we propose a more efficient approach which works
  directly with the variables and their inverse. We show that the
  commutation relations and the inversion relations characterize toric
  border bases. We explicitly describe the first syzygy module
  associated to a toric border basis in terms of these
  relations. Finally, a new border basis algorithm for Laurent
  polynomials is described and a proof of its termination is given for
  zero-dimensional toric ideals.
\end{abstract}

 
\section{Introduction}\label{sec:1}

Polynomial equations appear naturally in many applications, as a way
to describe constraints between the unknown variables of a problem.
These could be, for instance, geometric constraints between objects (as in robotics, or
CAGD) or physical constraints (as in chemistry). 
In such problems, partial additional information may also be known: the
unknown variables are not zero, or real, positive, between 0 and 1, etc.

Finding the solutions by exploiting these constraints and 
this additional information is thus an
important operation, which usually requires dedicated and efficient methods. 

An algebraic approach to get (all) the complex solutions of a
polynomial system is based on the computation of quotient algebra
structures \cite{CLO97,EM07}. 
These structures are described effectively by a set of polynomials 
which represent the normal forms in the quotient structure and 
a method to compute the normal form of any polynomial.
This family of methods includes, for instance, Gr\"obner basis
\cite{CLO92, Eis94} or border basis computation \cite{BMnf99,BMPhT05,KK06}. 
A ``fixed-point'' strategy is involved in these algorithms: starting
with the initial set of equations, so-called $S$-polynomials or commutation
polynomials are computed and reduced. If non-zero remainders are
found, the set of equations is updated and the computation is
iterated; otherwise, the process is stopped.

An important difference between Gr\"obner bases and border Bases
is that a monomial ordering compatible with monomial
multiplication is necessary in the first type of methods. This
monomial ordering is used to define the initial ideal associated to
the ideal of the equations. The border basis approach extends Gr\"obner
basis methods by removing the monomial ordering constraint, which
may induce numerical instability when the coefficients of the
polynomials are known approximately \cite{BMnf99, BMPhT05,KK05,KR05,KK06,
  BM07, 
  MOURRAIN:2008:INRIA-00343103:1,
  Ka11, BMPhT12}.

In this paper, we consider systems of polynomial equations defining
(complex) solutions with non-zero coordinates. In this case, 
variables can be inverted and a natural setting for normal form
computation is the ring of the Laurent polynomials. 
The extension of Gr\"obner basis algorithm to Laurent polynomials is
difficult, due to the lack of monomial well-order and the fact that 
 monomial ideals are trivial. A classical way to handle this difficulty is to
introduce new variables $y_{i}$ for the inverse of the initial
variables $x_{i}$ and new relations $x_{i} y_{i}-1=0$ and to compute
with polynomials in this extended set of variables. Doubling the
number of variables and adding new relations usually 
significantly reduce the performance of algorithms whose complexity is 
at least exponential in the number of variables (even doubly
exponential in the worst case). 

In the context of elimination and resultant theory, the 
approach of Macaulay \cite{Mac02} for the construction of projective resultant
matrices has been extended successfully to toric resultant for Laurent
polynomials \cite{GKZ94,CAEm93,Emiris:1994:MBP:190347.190374,St:Elim,CaEm00,Dandrea02}.
By analyzing the support of the Laurent polynomials, resultant matrices 
of smaller size than for the projective resultant can be constructed.
This leads to more efficient algorithms to compute 
a monomial basis of the quotient algebra and the (toric) roots for a
square polynomial system,
provided that it is {\em generic} for its support.

Our motivation is to develop normal form algorithms that can be
performed with Laurent polynomials in the same type of complexity
bounds than for usual polynomials, using efficient sparse linear
algebra on Laurent polynomial spaces.

\noindent{}\textbf{Contributions.} 
We extend the border basis approach to Laurent polynomials and show
that a characterization similar to the criterion for classical border
bases \cite{BMnf99,BMPhT05,MOURRAIN:2008:INRIA-00343103:1} applies in
this case: namely, the commutation relations (the product of two
variables should commute) and the inversion relations (the product of
a variable by its inverse should be $1$) can be used to check if a set
of polynomials is a toric border basis in a given degree.  As a new
result, this characterization yields an explicit description of the
first module of syzygies of a toric border basis, generalizing in a
natural way results from \cite{MOURRAIN:2008:INRIA-00343103:1}.  It is
an extension of Schreyer Theorem which describes generators of the
first syzygy module of a Grobner basis in terms of the reduction of
$S$-polynomials \cite{Eis94}[Theorem 15.10]. We also deduce a new
algorithm for computing a toric border basis, which requires light
modifications of the classical border basis algorithm. We prove its
termination in the case of zero-dimensional toric ideals.

\noindent{}\textbf{Content.} The paper is organized as follows. 
Section 2 describes the normal form criterion for toric border bases.
In Section 3, an explicit family of generators of the first syzygy module 
of a toric border basis is given.
In Section 4, we detail the algorithm for computing a border
basis for Laurent polynomials and prove its correctness for
zero-dimensional ideals, before the concluding Section 5. 

\noindent{}\textbf{Notation.}\label{sec:2}
Let $\Mon$ be the set of Laurent monomials in the variables $x_{1},\ldots,
x_{n}$. 
 An element of $\Mon$ is of the form
$\xb^{\alpha}=x_{1}^{\alpha_{1}}\cdots x_{n}^{\alpha_{n}}$ with
$\alpha= (\alpha_{1}, \ldots, \alpha_{n}) \in \ZZ^{n}$.
\begin{definition} \label{def:deg}
The degree of a monomial $\xb^{\alpha}=x_{1}^{\alpha_{1}}\cdots$ $
x_{n}^{\alpha_{n}}\in \Mon$  is $\delta (\xb^{\alpha})= |\alpha_{1}|+ \cdots +
|\alpha_{n}|$, which we also denote $\delta (\alpha)$.
\end{definition}
We will use the following notation: $x_{-i}=x_{i}^{-1}$, $i=1\ldots
n$, $x_{0}=1$, $[-n,n]^{*}=\{i \in [-n,n] \mid i\neq 0\}$.
We say that a sequence $(i_{1}, \ldots, i_{k}), i_{j}\in [-n,n]^{*}$ is canonical if
$|i_{1}| \leq |i_{2}| \leq \cdots \leq |i_{k}|$ and 
$\delta (x_{i_{1}}\cdots x_{i_{k}})=k$. It corresponds to a canonical
way to write a monomial as a product of variables.

For $m\in \Mon$, we denote by $((m))$ the cone generated by $m$, that
is, $((m))=\{ m'\in \Mon \mid m'= m\,m''\ s.t.\ m''\in \Mon, 
\delta (m')=\delta (m)+\delta(m'') \}$. It corresponds to the set of monomial
multiples of $m$, which are in the same ``quadrant'' of $\ZZ^{n}$.

Let $S=\kx$ be the ring of Laurent polynomials in the variables $x_{1},\ldots,
x_{n}$ with coefficients in a field $\kk$, that is the set of finite
linear combinations of monomials in $\Mon$.

For $p=\sum_{\alpha\in A} p_{\alpha}\, \xb^{\alpha} \in S$ with $p_{\alpha}\neq 0$, $A$ is the
support of $p$ and $\delta (p) = \max_{\alpha \in A} \delta (\alpha)$.

 
For $F\subset S$, 
let $\vspan{F}$ be the $\kk$-vector space spanned by $F$.

For $d \in \NN$ and $F\subset S$, let $F_{\le d}$  (resp. $F_{d}$) be
the set of polynomials $p\in F$ such that $\delta (p)\le d$
(resp. $\delta (p)= d$).

For $d\in \NN_{+}=\NN \setminus \{0\}$, let $\setdeg{F}{\le d}=\{ m\,f \mid m\in \Mon, f\in F,
\delta (m\,f)\le d \}$ and $\setdeg{F}{d}= \setdeg{F}{\le d}
\setminus \setdeg{F}{\le d-1}$.

For $B\subset \Mon$, we denote $B^{\times}= B\cup x_{1}B\cup \cdots \cup
x_{n}B\cup x_{1}^{-1}B\cup \cdots \cup x_{n}^{-1}B$ and call it the prolongation
of $B$. Let $\partial B = B^{\times}\setminus B$ be the border of $B$. 
Let $B^{[0]}=B$ and for $k\in \NN_{+}$, let $B^{[k]}= (B^{[k-1]})^{\times}$.

A set $B\subset \Mon$ is {\em connected to $1$} if $1\in B$ and
$\forall m \in B\setminus \{1\}$, there exists $i\in [-n,n]^{*}$ and
$m'\in B$, such that $m=x_{i}\, m'$ and $\delta (m')< \delta (m)$.

\section{Normal form criterion}\label{sec:3}

In this section, we describe normal form criteria for toric border
bases. The purpose of these criteria is to determine when we have a
decomposition:
$$ 
S_{\le d} = \Span{B}_{\le d} \oplus (F)_{\le d}
$$
for a monomial set $B$ connected to $1$, a polynomial set $F$ and a
degree $d\in \NN$.
The conditions that we describe extend naturally those known for classical
border bases, which are related to the commutation property of multiplication operators. 

Hereafter, $B\subset \Mon$ is a finite set of monomials connected
to $1$, $V=\Span {B}$ and $V^{\times}= \Span{B^{\times}}$.

Let $\pi:
V^{\times}_{\le d} \rightarrow V_{\le d}$ be a projection such that
$\pi \circ \pi = \pi$ and $\pi_{|V_{\le d}}$ is the identity map, and which is
compatible with the degree $\delta$: $\forall b \in \vspan{B^{+}}_{\le d}$,
$\delta (\pi (b)) \le \delta (b)$.

Let $K=\ker \pi$ be the kernel of $\pi$, so that
$$ 
V^{\times} = V \oplus K.
$$
As each monomial of $(\partial B)_{\leq d}$ can be projected by $\pi$
in $V_{\leq d}$, we can define a rewriting family as follows: 
\begin{definition}
For $B\subset \Mon$ connected to $1$ and a projection $\pi:
\Span{B^{\times}}_{\le d} \rightarrow \Span {B}_{\le d}$, the {\em rewriting family} for $\pi$ is the
set $F$ of polynomials of $\ker \pi$ of the form
\begin{equation}\label{eq:rewriting}
f_{\alpha} = \xb^{\alpha} - b_{\alpha}, 
\end{equation}
with $b_{\alpha}=\pi (\xb^{\alpha}) \in \Span{B}_{\le d}$, $\xb^{\alpha} \in (\partial B)_{\le d}$.
\end{definition}
We check that $F$ is a generating set of $K=\ker \pi$. Conversely,
a set $F$ of polynomials of the form \eqref{eq:rewriting} with 
$b_{\alpha} \in \Span{B}_{\le d}$, 
$\xb^{\alpha} \in (\partial B)_{\le d}$ and 
$\Span{B^{\times}}_{\le  d}= \Span{F} \oplus \Span{B}$
defines a projection from $\Span{B^{\times}}_{\le  d}$ 
onto $\Span {B}_{\le d}$.

\begin{definition}
For $F\subset S$ and $B\subset \Mon$, 
let $\com_{B} (F)$ be the set of polynomials in
$\vspan{B^{\times}}$  which are of the form
\begin{enumerate}
 \item $x_{i} f$ for some $f\in F$, $i\in [-n,n]^{*}$ or 
 \item $x_{i} f - x_{j} f'$ for $f,f'\in F$, $-1\le i<j \le n$.
\end{enumerate}
 The set of polynomials of type $1$ (resp. $2$) is denoted $\com_{B}^{1} (F)$
 (resp. $\com_{B}^{2} (F)$). 
The polynomials in $\com_{B}^{1} (F)$ (resp. $\com_{B}^{2} (F)$) 
are called the prolongation (resp. commutation) polynomials of $F$ for $B$.
Let $\com_{B} (F)= \com_{B}^{1} (F)\cup \com_{B}^{2} (F)$.
\end{definition}
From this definition, we see that $\com (F)\subset \vspan{F^{\times}} \cap \vspan{B^{\times}}$.
Hereafter, the set $B$ will be fixed and we will
simply write $\com_{B} (F)= \com (F)$. 

We define the operator of multiplication by $x_{i}$ associated
to $\pi$ as:
\begin{eqnarray*}
X_{i}: \vspan{B}_{\le d-1}& \rightarrow &  \vspan{B}_{\le d}\\
 b & \mapsto & \pi (x_{i} b).
\end{eqnarray*}
As $\pi$ is compatible with the degree, the image by $X_{i}$ of 
an element of degree $\le k <d$ is of degree $\le k+1$.

For a monomial $\xb^{\alpha}=x_{1}^{\alpha_{1}}\cdots x_{n}^{\alpha_{n}} \in
\Mon$ of degree $\le d$, we define
$\Xb^{\alpha} := X_{1}^{\alpha_{1}}\circ \cdots \circ
X_{n}^{\alpha_{n}}$.
It is an operator from $\vspan{B}_{\le d-\delta (\alpha)}$ to $\vspan{B}_{\le d}$.
We extend this construction by linearity and for any $p \in S_{\le d}$, we define
$p (\Xb) : \vspan{B}_{\le d-\delta (p)} \rightarrow \vspan{B}_{\le d}.$

As $B$ contains $1$ which is of degree $0$, we can then define 
\begin{eqnarray*}
\sigma : S_{\le d}& \rightarrow &  \vspan{B}_{\le d}\\
 p & \mapsto & p (\Xb) (1).
\end{eqnarray*}
Its kernel is denoted $I_{\pi, d}$.


Our objective is to relate properties of commutation and inversion of
the operators $X_{i}$ with the property that $\sigma$ defines a normal
form, that is a projection on $\Span{B}_{\le d}$ along the ideal
generated by $F$ in degree $\le d$.
We also relate it with a property which is easy to test
algorithmically, namely the polynomials $\Cc_{B}(F)$
reduce to $0$ by the rewriting family $F$.

The techniques used here are very similar to those developed in
\cite{BMnf99, BMPhT05,MOURRAIN:2008:INRIA-00343103:1,BMPhT12}, but
they require specific adaptations to the toric case, which we need to
detail.

\begin{thm}\label{THM:BB:DEG:D}
Let $d\geq 2$, let $B$ be a subset of $\Mon$ connected to $1$, 
let $\pi :\vspan{B^{\times}}_{\le d} \rightarrow \vspan{B}_{\le d}$ be a
projection and let ${F}$ be the rewriting family for $\pi$.
The following conditions are equivalent:
\begin{enumerate}
 \item $(X_{i}\circ X_{-i})_{| \vspan{B}_{\le d-2}}=Id$ for $1 \leq i
   \leq n$, 

$(X_{i}\circ X_{j} -X_{j} \circ X_{i})_{| \vspan{B}_{\le d-2}}=0$ 
for $1\le i<j\le n$,
 \item the map $\sigma$ is a projection which defines the exact sequence 
$$
0\rightarrow  \spandeg{F}{\le d}
   \rightarrow \Sc_{\le d} \stackrel{\sigma}{\longrightarrow}\Span{B}_{\le d} \rightarrow 0
$$
\item $\forall r \in \com_{B} (F_{\le d-1})$, $\pi ( r)=0$.
\end{enumerate}
\end{thm}
\begin{proof}
$1) \Rightarrow 2):$ 
As the operators $X_{i}$ are commuting and $X_{-i}=X_{i}^{-1}$ in
degree $d-2$, for any monomials $m,m'$ such that $\delta (m)\le d, \delta (m')\le d,
\delta (m\,m')\le d$, we have $m (\Xb) \circ m' (\Xb)= m' (\Xb) \circ m (\Xb)$
and $\sigma (m\,m')= m (\Xb)( \sigma (m'))$.
The construction of $\sigma$ is independent of the order in which we compose the operators $X_{i}$
since they are commuting.

Let us show by induction on $\delta (m)$, that for all monomials $m\in
B^{\times}$, we have $\sigma (m) = \pi (m)$.

The only monomial $m\in \Mon$ such that $\delta (m)=0$ is $1\in B$ and 
by definition $\sigma (1)=\pi (1) = 1$. The property is true for the degree $0$.

Assume that it is true in degree $0 \le k-1< d$ and let $m\in B$ with
$\delta (m)=k$.
As $B^{\times}$ is connected to $1$, there exists $i\in [-n,n]^{*}$ and $m'\in
B$ with $\delta (m') \le k-1$ such that $m=x_{i}\, m'$. 
By induction, we have $\sigma (m')=\pi (m')=m'$, thus 
$$
\sigma (m) = m (\Xb) (1) = X_{i} (\sigma (m')) = X_{i} (m') = \pi (x_{i} \,m') = \pi (m).
$$

This shows, in particular, that $\forall m\in B, \sigma (m)=m$
and that $\sigma \circ \sigma = \sigma$.
We deduce that the image of $\sigma$
is $\Span{B}_{\le d}$ and that the kernel $I_{\pi,d}$ of $\sigma$ is generated 
by $p-\sigma (p)$ for $p\in S_{\le d}$.

We now prove that $\spandeg{F}{\le d} \subset I_{\pi,d}$. 
For any $m\in \partial B$, we have $\sigma (m)= \pi (m)$ and $\sigma (\pi
(m))= \pi (m)$ since $\pi (m)\in \Span{B}_{\le d}$.
This implies that $\sigma (m-\pi (m))=0$. 
We have shown that the elements $m-\pi (m)$,
$m\in \partial B$ are in $I_{\pi,d}$. We deduce that $F\subset
I_{\pi, d}$ and thus that $\spandeg{F}{\le d}\subset I_{\pi,d}$.

In the next step, we prove that $I_{\pi,d} \subset \spandeg{F}{\le d}$.
As $I_{\pi,d}$ is spanned by $p-\sigma (p)$ for $p\in S_{\le d}$, it
is sufficient to prove that for a monomial $m$ with $\delta
(m)\le d$, we have $m-\sigma (m)\in \spandeg{F}{\le  d}$, which we do by
induction on $\delta (m)$. The case $\delta (m)=0$ or $m=1$ is obvious.
For any monomial $m\in \Mon_{\le d}$, 
we can decompose it as $m=x_{i} m'$ with $i\in [-n,n]^{*}$ and $\delta
(m')< \delta (m)$. By the induction hypothesis, $m'-\sigma (m') \in
\spandeg{F}{\le  d-1}$. We deduce that
$$ 
m-\sigma (m) = x_{i} (m'-\sigma (m')) + x_{i} \sigma (m') - \pi (x_{i} \sigma (m'))
$$
is in $\spandeg{F}{\le  d}$, since $x_{i} \sigma (m') - \pi (x_{i} \sigma
(m')) \in \ker \pi= \Span{F}$.
This proves that $I_{\pi,d} \subset \spandeg{F}{\le d}$. 

This implies that $I_{\pi,d}= \spandeg{F}{\le d}$ and proves point (2).

\noindent{}$2) \Rightarrow 3):$
Let $r \in \com (F_{\le d-1})$ then $r \in  \vspan{( F_{\le
    d-1})^{\times}}\cap \vspan{B^{\times}}$.
As $\vspan{( F_{\le d-1})^{\times}}\subset \spandeg{F}{\le d} = \ker \sigma$
we have $\sigma (r)=0$. But $\sigma$ coincides with $\pi$ on
$\vspan{B^{\times}}$ so that we have $\pi (r)=0$, which shows that $r \in 
\ker \pi = \vspan{F}$.

\noindent{}$3) \Rightarrow 1):$
Let $m\in B$ of degree $\le d-2$ and $1 \le i < j \le n$. Suppose that
$m_{1} := x_{i} m \in \partial B$ and $m_{2} := x_{j} m \in \partial B$.
Let $f_{1} = m_{1} -\pi (m_{1}), f_{2}= m_{2} -\pi (m_{2}) \in F_{\le d-1}$.
As $x_{i} m_{2} = x_{j} m_{1} = x_{i} x_{j} m$, we have
\begin{eqnarray*}
{(X_{i} \circ X_{j} - X_{j} \circ X_{i}) (m)} 
 & = & \pi ( x_{i} \pi (m_{2})) - \pi ( x_{j} \pi (m_{1})) \\
 & = & \pi ( x_{i} (m_{2} - f_{2}) -  x_{j} (m_{1} -f_{1})) \\
 & = & \pi ( x_{j} f_{1} - x_{i} f_{2}).
\end{eqnarray*}
As $x_{j} f_{1} - x_{i} f_{2} =  x_{i}
\pi (m_{2}) -  x_{j} \pi (m_{1})\in \com_{B}(F_{\le d-1})$,
the hypothesis (3) implies that $\pi ( x_{j} f_{1} - x_{i} f_{2})=0$.
A similar argument applies if $x_{i} m \in B$ or $x_{j} m \in B$. 
Consequently, we have 
$(X_{i}\circ X_{j} -X_{j} \circ X_{i})_{| \vspan{B}_{\le d-2}}=0$.

Similarly, we have 
\begin{eqnarray*}
{X_{-i} \circ X_{i} (m)} 
 & = & \pi ( x_{-i} \pi (m_{1})) =  \pi ( x_{-i} (m_{1} -f_{1})) \\
 & = & \pi (m - x_{-i} f_{1})).
\end{eqnarray*}
As $m -x_{-i} f_{1}\in \Span{B^{\times}}$, we have $x_{-i} f_{i} \in \Span{B^{\times}}$
and thus $x_{-i} f_{1} \in \com_{B}(F_{\le d-1})$, which implies that  $\pi ( x_{-i} f_{1})=0$.
We deduce that $\pi ( m - x_{-i} f_{1})= \pi (m)= m$. A similar
argument applies if $x_{i} m \in B$. This proves that $(X_{-i}\circ
X_{i})_{| \vspan{B}_{\le d-2}}=Id$ and concludes the proof of point (3).
\end{proof}
If one of these (equivalent) conditions is satisfied, we say that $F$ is a {\em border basis} in degree $d$ for $B$.

\begin{rem}\label{rem:1}
If $F$ is a border basis for $B$ in degree $d$, then Theorem \ref{THM:BB:DEG:D}
implies that $F_{\le d'}$ is a border basis for $B$ in degree $d'$ for
any $2\le d'\le d$.
\end{rem}

\begin{rem}\label{rem:2}
If Theorem \ref{THM:BB:DEG:D} (2) is satisfied, then 
any element $p\in S_{\le d}$ is the sum of $\sigma (p)\in
\Span{B}_{\le d}$ and $p-\sigma (p)\in I_{\pi,d}= \spandeg{F}{\le d}$.
Moreover, $p\in \Span{B}_{\le d} \cap \spandeg{F}{\le d}$ is such that $p=\sigma (p)=0$.
We deduce that 
$$ 
S_{\le d} = \Span{B}_{\le d} \oplus  \spandeg{F}{\le d},
$$
and $\sigma$ is the projection on $ \Span{B}_{\le d}$ along $ \spandeg{F}{\le d}$.
It is also called a normal form on $S_{\le d}$ modulo $\spandeg{F}{\le d}$.
\end{rem}

\section{Syzygies}

Let $F$ be a border basis in any degree for a finite set $B$ of monomials, which is
connected to $1$.

For any monomial $m\in \Mon$, we define $\delta_{B} (m)$ as the
smallest integer $d\in \NN$ such that $m\in B^{[d]}$. If $m\in B$,
then $\delta_{B} (m)=0$. For any $m_{1},m_{2}\in \Mon$, $\delta_{B}
(m_{1}\, m_{2})\leq \delta (m_{1}) \delta_{B} (m_{2})$.

For any $i \in [-n,n]^{*}$, let $\mu_{i}$ be the multiplication by
$x_{i}$ in $S$. We define the map
\begin{eqnarray*} 
\psi_{i}: \Span{B} & \rightarrow & \Span{F}\\
           m    & \mapsto     & (\mu_{i} -X_{i}) (m) = x_{i} m - \pi(x_{i} m)
\end{eqnarray*}
For a monomial $m \in B$,  if $x_{i} m\in B$ then $\psi_{i} (m)=0$,
otherwise $\psi_{i} (m)$ is an element of $F$. Conversely, for any
$f\in F$ of the form $f=\xb^{\alpha}-b_{\alpha}$ with
 $\xb^{\alpha} \in \partial B$ and $b_{\alpha}\in \Span{B}$, there
 exist $m \in B$ and $i\in [-n,n]^{*}$ such that $\xb^{\alpha}= x_{i}
 m$. We deduce that $f=\psi_{i} (m)$. Therefore, the set of elements $\psi_{i}
 (m)\neq 0$ with $m\in B$, $i\in [-n,n]^{*}$ is $F$.

Using the relations between elements in $F$ and elements of the form
$\psi_{i} (m)$, we are going now to associate
to the set $F$ a basis of a free $S$ module.
The purpose of this construction is to describe generators
of the syzygies between the elements of $F$ explicitly.
We denote by $Y_{i}, i\in [-n,n]^{*}$ the canonical basis of the
vector space $Y= \kk^{2n}$.
Let $\Sc_{1}$ be the free $S$-module generated by 
$ Y_{i} \otimes m$
with $i \in [-n,n]^{*}$, $m\in B$ and 
$x_{i} m \not \in B$.
The basis of the $S$ module $\Sc_{1}$ is also denoted 
$$ 
Y_{i} [ m] :=  Y_{i} \otimes m.
$$
By convention, $Y_{i}[m]=0$ if $x_{i}\, m\in B$ and for any $b=
\sum_{j} \lambda_{j} m_{j}\in
\Span{B}$, $Y_{i}[b]= Y_{i}\otimes b = \sum_{j} \lambda_{j} Y_{i}[m_{j}]$.
An element of $\Sc_{1}$ is a sum of terms of the form 
$\lambda \, m_{1} Y_{i}[m_{2}]$ with $\lambda \in \kk\setminus \{0\}$, $m_{1}\in \Mon$, $m_{2}\in B$.

We extend the degree $\delta$ to $\Sc_{1}$ as follows: for any term 
of the form $m_{1} Y_{i}[m_{2}]$ with $m_{1}\in \Mon$, $m_{2}\in B$,
we set $\delta (m_{1} Y_{i}[m_{2}])= \delta (m_{1})$. For all
$r \in \Sc_{1}$, $\delta (r)$ is the maximum degree of its non-zero terms.

We define now the map $\partial_{1}: \Sc_{1} \rightarrow S$ as 
\begin{eqnarray*} 
\partial_{1}: \Sc_{1} & \rightarrow & S\\\;
Y_{i}[ m] & \mapsto & \psi_{i} (m).
\end{eqnarray*}
The kernel of $\partial_{1}$ is the set of syzygies
between the elements $\psi_{i} (m)=f$ of $F$.

\begin{lem}
$\forall m=x_{i_{1}} \cdots x_{i_{k}} \in \Mon$, we have 
$
m = \pi (m) + \partial_{1} (\Psi_{i_{1}, \ldots  ,{i_{k}}}) 
$
where
\begin{equation}\label{eq:Psi}
\Psi_{i_{1}, \ldots  ,{i_{k}}} = \sum_{l=1}^{k} x_{i_{1}} \cdots
x_{i_{l-1}} Y_{i_{l}} [ X_{i_{l+1}} \circ \cdots \circ X_{i_{k}} (1)].
\end{equation}
\end{lem}
\begin{proof} 
We prove the relation by induction on $k$. For $k=1$ and $i_{1} \in [-n,n]^{*}$,
$m=x_{i_{1}} = X_{i_{1}} (1) + \psi_{i_{1}} (1)$.
 
Assume that the property is true for $x_{i_{2}} \cdots x_{i_{k}} \in \Mon$.
Then, by induction hypothesis, 
\begin{eqnarray*}
\lefteqn{x_{i_{1}}\, x_{i_{2}} \cdots x_{i_{k}}= x_{i_{1}} (x_{i_{2}}
  \cdots x_{i_{k}}) } \\
& = & x_{i_{1}} (\pi(x_{i_{2}} \cdots x_{i_{k}}) + \partial_{1}(\Psi_{i_{2}, \ldots  ,{i_{k}}})) \\
&= &x_{i_{1}} X_{i_{2}}\circ \cdots \circ X_{i_{k}} (1) + x_{i_{1}} \sum_{l=2}^{k} x_{i_{2}} \cdots x_{i_{l-1}} \psi_{i_{l}} \circ 
X_{i_{l+1}} \circ \cdots \circ X_{i_{k}} (1)\\
& = &  X_{i_{1}} \circ X_{i_{2}}\circ \cdots \circ X_{i_{k}} (1)\\&&
+ \psi_{i_{1}}\circ  X_{i_{2}}\circ \cdots \circ X_{i_{k}} (1)
+ \sum_{l=2}^{k} x_{i_{1}} \cdots x_{i_{l-1}} \psi_{i_{l}} \circ X_{i_{l+1}} \circ \cdots \circ X_{i_{k}} (1). \\
& = &  X_{i_{1}} \circ X_{i_{2}}\circ \cdots \circ X_{i_{k}} (1)
+ \sum_{l=1}^{k} x_{i_{1}} \cdots x_{i_{l-1}} \psi_{i_{l}} \circ
X_{i_{l+1}} \circ \cdots \circ X_{i_{k}} (1).\\
& = &
\pi (x_{i_{1}}\cdots x_{i_{k}}) 
+
\partial_{1} (\sum_{l=1}^{k} x_{i_{1}} \cdots x_{i_{l-1}} Y_{i_{l}}[
X_{i_{l+1}} \circ \cdots \circ X_{i_{k}} (1)])
\end{eqnarray*}
since by Theorem \ref{THM:BB:DEG:D}, 
$\pi (x_{i_{1}}\cdots x_{i_{k}}) 
= \sigma (x_{i_{1}}\cdots x_{i_{k}}) 
= X_{i_{1}} \circ X_{i_{2}}\circ \cdots \circ X_{i_{k}} (1).
$ 
\end{proof}

This construction allows us to relate any term of $\Sc_{1}$ with an
element of the form $\Psi_{i_{1}, \ldots, i_{k}}$ as follows:
\begin{lem}\label{lem:term:Psi}
For any term $m_{1} Y_{i}[m_{2}]$ of $\Sc_{1}$ with $m_{1}\in \Mon$, $m_{2}\in B$,
there exists $i_{1}, \ldots, i_{l} \in [-n,n]^{*}$ such that 
$$
{\Psi_{i_{1}, \ldots, i_{k}}
=  m_{1} Y_{i}[ m_{2}]+r} \\
$$
with $\delta (r) < \delta (m_{1}) = \delta (m_{1} Y_{i}[m_{2}])$.
\end{lem}
\begin{proof}
Let $m_{1}= x_{i_{1}} \cdots x_{i_{k_{1}}}$ with $\delta
(m_{1})=k_{1}$ and $m_{2} = x_{j_{1}}\cdots
x_{j_{k_{2}}}$ with $ x_{j_{l}}\cdots x_{j_{k_{2}}} \in B$ for $1\le
l\le k_{2}$. 
As $Y_{j_{l}}[ X_{j_{l+1}}\circ \cdots \circ X_{j_{k_{2}}} (1)]
=Y_{j_{l}}[ x_{j_{l+1}} \cdots x_{j_{k_{2}}}]=0$, the expansion
\eqref{eq:Psi} of $\Psi_{i_{1}, \ldots, i_{k_{1}}, i, j_{1} ,\ldots, j_{k_{2}}}$
yields
$${\Psi_{i_{1}, \ldots, i_{k_{1}}, i, j_{1} ,\ldots, j_{k_{2}}}
=  m_{1} Y_{i}[ m_{2}]+r} \\
$$
with 
$$
r=  \sum_{l=1}^{k_{1}} x_{i_{1}} \cdots
x_{i_{l-1}} Y_{i_{l}}[ X_{i_{l+1}} \circ \cdots \circ X_{i_{k_{1}}} \circ X_{i}
\circ m_{2} (X) (1)].
$$
The term $r$ is such that $\delta (r)\le k_{1}-1< \delta (m_{1})$,
which proves the lemma.
\end{proof}

\begin{lem}
$\forall i \neq j\in [-n,n]^{*}$ and $\forall m\in B$, the element
\begin{equation*}\label{eq:genK1:phi} 
\phi_{i,j} (m) := x_{i} Y_{j}[m] - x_{j} Y_{i} [m] - Y_{j}[X_{i}(m)] + Y_{i}[X_{j}(m)] 
\end{equation*}
is in $\ker \partial_{1}$. 
\end{lem} 
\begin{proof}
By definition of $\psi_{i}$, we have 
\begin{eqnarray*}
\lefteqn{\partial_{1} (x_{i} Y_{j}[m] - x_{j} Y_{i} [m] - Y_{j}[X_{i}(m)] + Y_{i}[X_{j}(m)])} \\
 & = & x_{i} \psi_{j} (m) -x_{j} \psi_{i} (m) - \psi_{j} (X_{i} (m)) +
 \psi_{i} (X_{j} (m))\\
 & = & x_{i} (x_{j} m - X_{j} (m) ) - x_{j} (x_{i} m - X_{i} (m))\\
 &   & - (x_{j} X_{i} (m) - X_{j} (X_{i} (m))) 
+ (x_{i} X_{j} (m) - X_{i} (X_{j} (m)))\\
 & = &   X_{j} (X_{i} (m)) - X_{i} (X_{j} (m)) = 0
\end{eqnarray*}
since $X_{i}$ and $X_{j}$ commute.
\end{proof}
By linearity, we extend the map $\phi_{i,j}$ to the vector space
$\Span {B}$ spanned by $B$, so that $\phi_{i,j} (\sum_{k} \lambda_{k}
m_{k})= \sum_{k} \lambda_{k} \phi_{i,j} (m_{k})$.

\begin{lem}
$\forall i \in [-n,n]^{*}$ and $\forall m\in B$, the element
\begin{equation*}\label{eq:genK1:rho}
\rho_{i} (m) := x_{i} Y_{-i}[m] + Y_{i}[X_{-i}(m)] 
\end{equation*}
is in $\ker \partial_{1}$.
\end{lem}
\begin{proof}
\begin{eqnarray*}
\lefteqn{\partial_{1} (x_{i} Y_{-i}[m] + Y_{i}[X_{-i}(m)] )}\\
 & = & x_{i} \,\psi_{-i} (m) + \psi_{i} (X_{-i} (m)) \\
 & = & x_{i}\, ( x_{-i}\, m - X_{-i} (m) ) + (x_{i}\, X_{-i} (m) -X_{i}
 (X_{-i} (m)))\\
 & = & m - X_{i} \circ X_{-i} (m) =0
\end{eqnarray*}
since $X_{i}\circ X_{-i} = Id$.
\end{proof}
 
Let $K_{1}\subset \Sc_{1}$ be the $S$-module generated by the elements 
$\rho_{i} (m)$, $\phi_{i,j} (m)$ for $i \neq j\in [-n,n]^{*}$ and $m\in B$.

We are going now to describe how a term $m_{1} Y_{i}[m_{2}]$ with $m_{1}\in \Mon$,
$m_{2}\in B$ can be transformed modulo $K_{1}$. 

\begin{lem}\label{lem:equiv}
$\forall m=x_{i_{1}} \cdots x_{i_{k}} = 
x_{j_{1}} \cdots x_{j_{k'}} \in \Mon$,
$$ 
\Psi_{i_{1}, \ldots, i_{k}} -\Psi_{j_{1}, \ldots, j_{k'}}\in K_{1}.
$$
\end{lem}
\begin{proof}
By successive permutations of two adjacent indices and contraction of
adjacent indices $-i, i$, we
can transform any sequence $J=({j_{1}},\ldots, {j_{k'}})$
into a canonical sequence $I=(i_{1}, \ldots, i_{k})$.
Thus it is enough to prove the property for the permutation of
two consecutive indices: $J= (i_{1}, \ldots, i_{l}, i_{l+1},
\ldots, i_{k})$ and for the contraction of two indices 
$J=  (i_{1}, \ldots, i_{l}$, $-j, j, i_{l+1},
\ldots, i_{k})$.
By definition of $\Psi$ and since the operators $X_{i}$ are commuting, we have
\begin{eqnarray*} 
\lefteqn{\Psi_{\ldots, i_{l}, i_{l+1}, \ldots}
- 
\Psi_{\ldots, i_{l+1}, i_{l}, \ldots}} \\
&= & x_{i_{1}} \cdots x_{i_{l-1}} \left(x_{i_{l}} Y_{i_{l+1}} [ 
X_{i_{l+2}} \circ X_{i_{l+3}} \cdots \circ X_{i_{k}} (1)]\right.
 - x_{i_{l+1}} Y_{i_{l}} [ X_{i_{l+2}} \circ X_{i_{l+3}} \cdots \circ X_{i_{k}} (1)]\\
&& - Y_{i_{l}}[  X_{i_{l+1}} \circ X_{i_{l+2}} \circ X_{i_{l+3}}
\cdots \circ X_{i_{k}} (1)]
\left. + Y_{i_{l+1}} [ X_{i_{l}} \circ X_{i_{l+2}} \circ X_{i_{l+3}} \cdots \circ X_{i_{k}} (1)]\right)\\
& = &  x_{i_{1}} \cdots x_{i_{l-1}} \phi_{i,j} (X_{i_{l+2}} \circ X_{i_{l+3}} \cdots \circ X_{i_{k}} (1)) 
\end{eqnarray*} 
which is an element of $K_{1}$.
Similarly, for $j\in [-n,n]^{*}$ 
\begin{eqnarray*}
\lefteqn{\Psi_{\ldots, i_{l},j,-j, i_{l+1}, \ldots}
- \Psi_{\ldots, i_{l}, i_{l+1}, \ldots}} \\
&= & x_{i_{1}} \cdots x_{i_{l}} 
\left( Y_{j} [ X_{-j} \circ X_{i_{l+1}} \circ X_{i_{l+3}} \cdots \circ X_{i_{k}} (1)]\right.\\
& & + \left. x_{j} Y_{-j} [ X_{i_{l+1}} \circ X_{i_{l+3}} \cdots \circ X_{i_{k}} (1)]\right)\\
& = &  x_{i_{1}} \cdots x_{i_{l-1}} \rho_{j} (X_{i_{l+1}} \circ X_{i_{l+3}} \cdots \circ X_{i_{k}} (1))
\end{eqnarray*}
which is also an element of $K_{1}$.
\end{proof} 

\begin{thm}
The first module of syzygies of $F$ is generated by the elements 
\begin{itemize}
 \item $\rho_{i} (m)= x_{i} Y_{-i}[m] + Y_{i}[X_{-i}(m)]$, 
 \item $\phi_{i,j} (m)= x_{i} Y_{j}[m] - x_{j} Y_{i} [m] + Y_{i}[X_{j}(m)] - Y_{j}[X_{i}(m)]$ 
\end{itemize}
for $i \neq j\in [-n,n]^{*}$ and $m\in B$.
\end{thm}
\begin{proof} 
Let $s\in \ker \partial_{1}$ be a sum of non-zero terms of the form 
$\lambda m_{1} Y_{i}[m_{2}]$ with $\lambda\in \kk\setminus \{0\}$,
$m_{1}\in \Mon$, $m_{2} \in B$.

The monomial $m=m_{1} x_{i} m_{2}$ can be decomposed in a unique way 
as $m=m'_{1} x_{i'} m'_{2}$ with $\delta_{B} (m)=\delta (m'_{1})$ and
$m'_{2}\in B$ and $(i', m'_{1})$ the smallest possible for the lexicographic ordering.

By Lemma \ref{lem:term:Psi} and Lemma \ref{lem:equiv}, there exist
$i_{1},\ldots, i_{k}$, $i'_{1}, \ldots$, $i'_{k'} \in [-n,n]^{*}$
with 
$m = x_{i_{1}}\cdots x_{i_{k}}= x_{i'_{1}}\cdots x_{i'_{k'}}$ 
and $r\in \Sc$ with $\delta (r) < \delta (m_{1})$ such that
\begin{eqnarray*}
{m_{1} Y_{i}[m_{2}]}
&=& m'_{1} Y_{i'}[ m'_{2}] + r + \Psi_{i_{1}, \ldots, i_{k}} -
\Psi_{i'_{1}, \ldots, i'_{k'}}\\
&\equiv&
m'_{1} Y_{i'}[ m'_{2}] + r  \mod K_{1}.
\end{eqnarray*}
Applying this relation inductively on the terms of $r$, any element
$m_{1}Y_{i}[m_{2}]$ can be reduced modulo $K_{1}$ to a sum of terms of the form 
$\lambda m'_{1} Y_{i'}[ m'_{2}]$ with $\delta_{B} (m'_{1} x_{i'}
m'_{2}) =\delta (m'_{1})$
and $(i, m'_{1})$ the smallest possible for the lexicographic
ordering.

Thus we can assume that the terms of the decomposition of $s\in \ker \partial_{1}$ satisfy
this property. Hereafter, we call this decomposition a canonical
decomposition of $s$.
Suppose that the canonical decomposition $s$ is not zero. 
Then for any term $(m_{1} Y_{i}[ m_{2}]$ of this canonical decomposition, we have 
$\partial_{1} (m_{1} Y_{i}[ m_{2}])= m_{1} x_{i} m_{2} + p$
with $\delta_{B} (p) < \delta_{B} (m_{1} x_{i} m_{2})$.

Let us consider a term $\lambda\, m_{1} Y_{i}[ m_{2}]$ such that
$\delta_{B} (m_{1}  x_{i}$ $ m_{2})$ $= \delta (m_{1})$ is maximal. 
As $s\in \ker \partial_{1}$, the monomial $\lambda m_{1} x_{i} m_{2}$ must be cancelled 
by a monomial  $\lambda' m'_{1}  x_{i'} m'_{2}$ from the image 
$\partial_{1} (m'_{1} Y_{i'}[m'_{2}])$ of a distinct term of the
canonical decomposition $s$.

As there is a unique way to decompose a monomial $m\in \Mon$ as
$m=m_{1} x_{i} m_{2}$ with 
$m_{1} Y_{i}[m_{2}] \neq 0$,
$\delta_{B} (m'_{1} x_{i'} m'_{2}) =\delta (m'_{1})$
and $(i, m'_{1})$ the smallest possible for the lexicographic
ordering, this is not possible. We obtain a contradiction, which shows that
the canonical decomposition of $s$ is zero and  that
$s\in \ker \partial_{1}$ can be reduced to $0$ modulo $K_{1}$. In
other words, $K_{1}= \ker \partial_{1}$ which proves the theorem.
\end{proof}

\section{Algorithm}\label{sec:4}

In this section we describe an algorithm based on the above
properties to compute a border basis for a system of Laurent
polynomials. For the sake of simplicity, and as they were not needed
before, we introduce the following definitions:

\begin{defn}
  The ball $Ball(k)$ of radius $k$ is the set of monomials $m$ of
  $\XX$ such that $\delta(m)\leq k$.
\end{defn}

\begin{defn}
  Let $f$ be a Laurent polynomial. A monomial $m$ of the support
  of $f$ is said to be {\em extremal} for $f$ if
  $\delta(m)=max(\delta(m'),\ m'\in\supp(f))$. In other words an
  extremal monomial is a monomial of maximal degree.
\end{defn}

We use the preceeding definitions to extend the notion of choice
functions introduced in \cite{BMPhT05} to the context of Laurent
polynomials. We recall that in the context of usual
polynomials, the choice function generalizes the construction of the leading monomial
for a monomial ordering. It is used to the select a
monomial of a polynomial in the reduction process.

\begin{defn}
  A choice function $\gamma$ refining the degree $\delta$
 is a function such that, given a Laurent polynomial $f$ 
  returns a monomial $\gamma(f)$ of the support of $f$ that is
  extremal.
\end{defn}

\subsection{Description}
We now describe the complete algorithm for computing a border basis for a
Laurent polynomial system. This algorithm follows the same approach as
in \cite{BMPhT12}, with adaptations to be done for
dealing with the fact that the prolongation operation $\cdot^\times$ can lead to degree drops.
It is a ``fixed-point'' method which updates a set of polynomials $F$ and
a monomial set $B$ until they stabilize. The update is done so that if a fixed-point is reached,
then $F$ is a border basis for the monomial set $B$.

The monomial set $B$ is represented as a finite union of differences of cones
$B=\cup_i \left ( ((m_{i})) \right.$ $\setminus ((m_{j_1}))$ $ \cdots$ $\left. \setminus
((m_{j_{k'_i}})) \right)$. By construction, if $1\in B$, then $B$ is
connected to $1$. When a monomial $m$ is removed from $B$,
the corresponding cone $((m))$ is removed from the representation of
$B$, so that $B$ remains connected to $1$.

\begin{algorithm}\caption{Border basis \label{algo:bbt}}  
\KwIn{A set of Laurent polynomials $\Fc=\{f_1,\ldots,f_s\}$ and
  $\gamma$ a choice function refining the degree.}
\KwOut{A border basis for $I=(f_1,\ldots,f_s)$.}
\begin{itemize}
 \item $[k,F_{k},B]:=Initialization(\Fc)$
 \item While $not(is\_empty(F_k)) \mathrm{\ or\ } k<\max_{f\in \Fc}\delta(f);$

 {\bf core loop} 
   \hspace{-10pt} \begin{enumerate}
    \item Compute $C^{1}_{k+1}:=\com_{B}^{1} (F_k)$ and
      $A_{k+1}:= (B^{\times})_{\le k+1}$.
    \item If there exists polynomials in $C^{1}_{k+1}$ of degree $<k$,
      then 
     \begin{itemize}
        \item Compute the minimal $l$ such that the ball of radius $l$
          contains polynomials of $C^{1}_{k+1}$;
        \item Using $\gamma$ choose leading monomials  for the
          polynomials of $\{f\in C^{1}_{k+1}, supp(f)\subset Ball(k)\}$;
        \item Compute $C^{1}_{l}\cup \{f\in C^{1}_{k+1}, supp(f)\subset Ball(k)\}$
        \item Set $k=l$;
      \end{itemize}
    \item Construct the matrix $M_{k+1}:=(C^{1}_{k+1}|A_{k+1})$.
    \item Compute $r_{k+1} := \rank M_{k+1}$.
   \item If $\vspan{C^{1}_{k+1}}$ contains polynomials of degree
    $<k+1$, add them to $\Fc$ and start a new loop with
    $k:=\min_{p\in C^{1}_{k+1}}\delta (p)$.
  \item If $\# (A_{k+1}\setminus B_{k+1})\neq r_{k+1}$,
\begin{itemize}
\item compute $A'_{k+1}\subset A_{k+1}$ such that $\# A'_{k+1}=
  r_{k+1}=\rank (F_{k+1}\mid A'_{k+1})$; for instance looking at the
  monomials indexing the columns of a maximal invertible submatrix of
  $M_{k+1}$.
 \item compute $B'_{k+1}= A_{k+1} \setminus A'_{k+1}$; 
 \item add the monomials $B'_{k+1}$ to $B$; 
\end{itemize}
    \item Define $\pi_{k+1}: \vspan{B^{\times}}_{\le k+1}\rightarrow
\vspan{B}_{\le k+1}$ as the extension of $\pi_{k}$ such that
$C^{1}_{k+1}\subset \ker \pi_{k+1}$ and $F_{k+1}$
as the new polynomials in the ball of radius $k+1$ in the corresponding rewriting family.
     \item Compute $C^{2}_{k+1} :=\pi_{k+1} (\Cc^2_B(F_k))\cup {\pi}_{k+1} (\Fc_{k+1})$.
     \item If 
$C^{2}_{k+1} = \{0\}$, then start a new loop with $k:=k+1$.
   \item If $\vspan{C^{2}_{k+1}}$ contains polynomials in the ball of radius
    $<k+1$, add them to $\Fc$ and start a new loop with
    $k:=\min_{p\in C^{2}_{k+1}}\delta (p)$.

     \item 

Apply $\gamma$ to $\vspan{C^{2}_{k+1}}$, remove the monomial ideal generated by
      $\gamma (\vspan{C^{2}_{k+1}})$ from $B$ and update $\pi_{k+1}: \vspan{B^{\times}}_{\le k+1}\rightarrow
\vspan{B}_{\le k+1}$ and $F_{k+1}$, so that $C^{2}_{k+1}\subset \ker \pi_{k+1}$.
    \end{enumerate}
  \end{itemize}
 \end{algorithm}

Let us detail how Algorithm \ref{algo:bbt} is running.
The procedure $Initialization(\Fc)$ is used to find
the initial degree $k$ and the initial relations $F_k$ of degree $k$
and the associated monomial set $B$. 

The core of the algorithm is a loop where the set of polynomials $F_{k}$
and the monomial set $B$ are updated. The variable $k$ of each loop is
the degree in which the polynomial operations are performed. 

After computing the prolongation polynomials $\Cc^{1} (F_{k})$ in step 1, the degree
$k$ is adjusted in step 2 to the maximal degree of these polynomials.

The coefficient matrix of these polynomials is computed in step 3, and
used in step 4 and 5 to determine their rank and the minimal degree
of a polynomial in the vector space that they span.
The degree $k$ is then adjusted to this minimal degree in a new loop.

Step 6 checks if there is a rank deficiency, that is, if the rank of the
prolongation polynomials $\Cc^{1} (F_{k})$ corresponds to the number of
new monomials of $\partial B$ in degree $k+1$. If there is a rank
deficiency, the monomial set $B$ is extended by new monomials so that there is no rank deficiency.

Step 7 constructs the projection and the rewriting family in degree
$k+1$.

The steps 8, 9 compute the commutation polynomials 
$\Cc^{2}(F_{k})$, the polynomial of degree $k+1$ of $\Fc$
and their remainders by the rewriting family in degree $k+1$.

The steps 10 and 11 checks if there polynomials of degree $<k+1$ in
the vector space spanned by these remainders and update the degree and
the monomial set $B$ if needed.

As we can see, the degree $k$ can either increase or drop to a lower
value (in steps 2, 5 and 10).

The degree $k$ is adjusted so that throughout the algorithm, the
leading monomial of the polynomials in $F_{k}$ is of degree $\le k$.

{}The new algorithm does not required a large modification of the
border basis algorithm of \cite{BMPhT12}. 

\subsection{Correctness}
We now prove that this algorithm stops and produces a border basis of
the generators of a zero dimensional ideal, that is, a system of Laurent polynomials
$\Fc$ such that $\dim \Ac< \infty$ where $\Ac=S/(\Fc)$. A proof relying on the Noetherianity
of monomial ideals (as for Gr\"obner bases) 
cannot be employed here since all non-zero monomial ideals of $S$ are equal to $(1)$.

We give here a lemma which is useful in the proof of correctness.

\begin{lem}\label{lem-ext}
  Let $k\in\NN$ then $\ker(\pi_k)^{\times}\subset \ker \pi_{k+1}$.
\end{lem}
\begin{proof}
  This comes from the fact that we construct $\pi_{k+1}$ as a
  prolongation of $\pi_k$ in step 7.
\end{proof}

We can now give the proof of termination and correctness of the
algorithm. This proof relies heavily on the fact that $\gamma$ refines
the degree. It is new and simpler than the one given in \cite{BMPhT05}.

\begin{thm}
If $\Fc$ is zero dimensionnal, i.e. $\dim S/(\Fc)< \infty$, then Algorithm
\ref{algo:bbt} stops and returns a border basis $\Fc$ of $\Fc$ for a
monomial set $B$. 
\end{thm}

\begin{proof}
  Let us first notice that if at anytime in the algorithm, the
  variable $k$ is equal to $0$ then
  the algorithm stops and returns the polynomial $1$, which allows us
  to define the null projection. 

  We now remark that the construction of $B$ is not monotonic: $B$ can
  increase or decrease. However, the monomial set $B$ is extended with
  new monomials only at step $6$. At step $6$, the
  algorithm is operating on polynomials of degree $k$, performing the
  prolongation operation $\cdot^x$ on them, 
  checking that no non-zero remainder of degree less than
  $k+1$ have appeared and applying linear
  algebra steps. If some monomials are added to $B$, these monomials
  are the leading monomials of polynomials of degree $k+1$ and hence
  they are of degree $k+1$. This means that one cannot add to $B$ a
  monomial of degree $1$, otherwise we would have gone through degree
  $0$. Hence there are finitely many {\tt core loop} turns where $k=1$.

  Let us now prove by induction that for all $d\in\NN$ there is
  finitely many {\tt core loop} turns where $k=d$. This is true for
  $d=1$. Let us suppose this is true for some $d$ and prove it for
  $d+1$. Consider one {\tt  core loop} turn {\em after} the last loop
  turn where $k=d$ (this last turn exists due to our hypothesis). Then as
  mentioned in the above paragraph, not going any further with $k=d$,
  one will never add to $B$ any monomial of degree $d+1$. We now remark
  that at the {\tt core loop} considered, there are two possibilities:
  \begin{itemize}
    \item $k$ is strictly below $d$ and due to our induction
      hypothesis one will never reach degree $d+1$ since the degree
      $d+1$ can only be reached after a step at degree $k=d$,
    \item $k$ is greater than $d$. In that case, $k$
      can become equal to $d+1$ only in case of a degree drop during
      subsequent {\tt core loop} turns, and that dropping to degree
      $d+1$ implies removing from $B$ at least one monomial of degree
      $d+1$. As there is only finitely many monomials of degree $d+1$,
      $k$ cannot take infinitely many times the value $d+1$.
  \end{itemize}
  In these two cases, $k$ can only be finitely many times equal to $d+1$
  which ends our induction.

  Let us show now that $k$ is bounded during all the computation. 

  First remark that the $B$ constructed by this algorithm is connected
  to $1$, hence if in the {\tt core loop} $k=d$, there is at least $d$
  monomials in $B$ at that loop turn. Let $D=\dim S/ (\Fc) < \infty$.
  If at any time $k$ becomes greater than $D+1$ then there exists a
  polynomial in the ideal $I$ whose support is included in the
  corresponding monomial set $B$. Let $p=p_1f_1+p_2f_2+\cdots+p_sf_s$
  be this polynomial. Applying Lemma \ref{lem-ext} inductively 
  starting from $max(deg(f_i),\ i\in [1,s])$ up to
  $k'=max(deg(p_if_i),\ i\in [1,s])$, one has that $p\in
  \ker(\pi_{k'})$. Hence $k'\geq k$ since $\pi_k$ is the identity on
  $B$. Since $p\in \ker(\pi_{k'})$, there is at least one monomial of
  the support of $p$ outside $B$ when the {\tt core loop} is ran with
  $k=k'$. Let $m$ be this monomial, one have $deg(m)\leq D+1$ since
  $B$ is connected to $1$. Therefore between the initial step
  when $k>D+1$ and the step when $k=k'$, $k$ must drop to a degree
  less than $D+1$. Since there is finitely many drops possible below
  degree $D+1$ there is finitely many moment when $k>D+1$.

  We deduce that $k$ remains bounded and for each degree $d\in \NN$,
  there is finitely many steps when $k=d$. This implies that the
  algorithm eventually stops.

  The termination of the algorithm follows from the inductive application
  of Lemma \ref{lem-ext}.  

  As the algorithm stops, say with $k=d$, all the monomials of $\partial B_{d}$ are
  the leading monomial of an element of the rewriting family $F_{d}$. 
  Since the monomial set $B_{d}$ is not updated during the
  last loop of the algorithm, the commutation polynomials $\Cc_{B}^{1}
  (F_{d})$, $\Cc_{B}^{2} (F_{d})$ project to $0$ by $\pi_{d+1}$.  By
  Theorem \ref{THM:BB:DEG:D}, we deduce that $F_{d}$ is a border basis
  for $B_{d}$.
\end{proof}

\subsection{Example}
\label{sec:5}

Let us examine the behavior of the previous algorithm on a generic
quadratic system of two equations in two variables, that is a system of
generic Laurent polynomials which support is the set of integer points
$A$ of the convex hull of $(-2,0), (0,-2), (0,2), (2,0) \in \ZZ^{2}$. 

Suppose that we use a Macaulay-like choice function, i.e., a function that chooses one
monomial of highest partial degree.

The {\em Initialization} procedure defines an initial monomial set $B$ to be all
the monomials except those of higher partial degree than $2$ which
graphically looks like:

\begin{center}
  \includegraphics{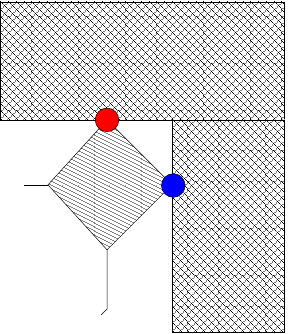}
\end{center}
The prolongation operation  $\,^\times$ on this initial configuration produces six new
polynomials, obtained by multiplying the equations by all the
variables that either follow the border of $B$ or get into it. We have
drawn there leading monomials and the Newton polytope of one of them.

\begin{center}
  \includegraphics{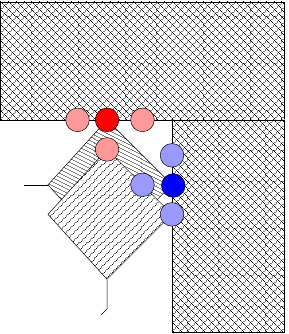}
\end{center}

According to our test in the core loop, the new polynomial drawn is
such that its leading monomial is not among the maximal degree
monomials of the support. The other polynomial is also in this
situation. So, according to the test done in step $2$, two new leading
monomials are chosen for these two polynomials and $B$ is update
accordingly.

\begin{center}
  \includegraphics{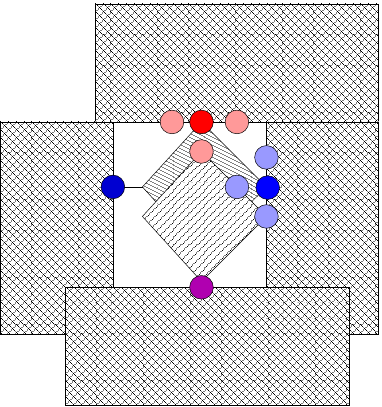}
\end{center}

The rest of the computation consists in following the border of this
monomial set $B$ in the same way as it is done in the generic polynomial
setting \cite{BMPhT00}.

The algorithm ends with a monomial set $B=\{ x_{1}^{\alpha_{1}}
x_{2}^{\alpha_{2}}\mid -2 \leq \alpha_{1} < 2, -2 \leq \alpha_{2} <
2\}$ of size $16$, which is the normalized volume of $A$ (i.e. $2 \times
Vol(A)$). The corresponding border basis $F$ is a set of 16
polynomials, one for each monomial of $\partial B$. 

What occurs on this example of a generic $2$ variate system can be generalized to $n$ variate
systems. The following table shows the sizes $N$ of the different
linear systems that have to be solved and the associated integer $k$
in the main loop of the algorithm, for solving a system of $n$ generic $n$ variate
Laurent polynomials of $\delta$-degree $2$ in each variable. 
For comparison, we give the size $M$ of the matrix to be inverted in the Schur complement
computation involved in the sparse resultant approach (see \cite{CAEm93}).  {\small
$$
\begin{array}{|c|c|c|}
  \hline
  n& N & M\\
  \hline
  2 & 2, 6, 6, 2 & 39\\
  3 & 3, 15, 30, 30, 15, 3 & 475\\
  4 & 4, 28, 84, 140, 140, 84, 28, 4  & 5165\\
  5 & 45, 180, 420, 630, 630, 420, 180, 45, 5 &54306\\
  6 & 6, 66, 330, 990, 1980, 2772, 2772, 1980, 990, 330, 66, 6 &566461\\
  \hline
\end{array}
$$
} 
In this table, we clearly see the improvement of the toric border basis
algorithm, compared to the sparse resultant method. Instead of solving  
one big linear system, the toric border basis computation involves the
solution of several much smaller linear systems. This improves both the
complexity and the numerical behaviour of the method. 

\section{Conclusion}

Normal form methods provides an effective way to compute the quotient
structure of a polynomial ring by an ideal, and thus to solve
polynomial equations. 
The Gr\"obner basis approach consists in completing a set of rewriting
rules on the monomials which is driven by a monomial ordering. 
Its extension to Laurent polynomials is difficult or expensive.

The border basis approach consists in imposing commutation relations
to operators of multiplication, extending the rewriting techniques to
a wider class of problems.  We show in this paper, that the approach
can naturally be extended to Laurent polynomials by imposing inversion
and commutation relations to the multiplication operators. The border
basis approach provides also a description of the first module of
syzygies.  This leads to a normal form algorithm for Laurent
polynomials, which performs linear algebra operations on monomials
with exponents in $\ZZ^{n}$.  

If the ideal $(\Fc)$ is a zero-dimensional ideal, we have 
shown the termination of the new algorithm. For ideal of positive dimension, we plan to
investigate techniques based on regularity detection as in \cite{BMPhT12}.

In this paper, we have considered Laurent polynomial rings, in which
all the variables are invertible. We can check that the approach
applies also to rings where only some of the variables are invertible,
by considering the inversion relations for these variables and the
commutation relations for all the pairs of variables.

This approach can be used to compute the solutions of a polynomial
system outside a variety: $g(x_{1}, \ldots, x_{n})\neq 0$. A new
invertible variable $x_{n+1}$ and the equation 
$x_{n+1} - g (x_{1}, \ldots, x_{n})=0$ can be introduced to compute the solutions of a
system outside the hypersurface defined by $g$.
We plan to investigate further applications of this toric border basis
approach such as residual intersections and to compare it with saturation techniques for classical
polynomial computation.
We also plan to provide an implementation of
this new algorithm in the package \textsc{borderbasix}\footnote{\small
 \texttt{mathemagix.org/www/borderbasix/doc/html/index.en.html}}.


\begin{thebibliography}{10}

\bibitem{CAEm93}
J.~Canny and I.~Emiris.
\newblock An efficient algorithm for the sparse mixed resultant.
\newblock In G.~Cohen, T.~Mora, and O.~Moreno, editors, {\em Proc.\ Intern.\
  Symp.\ Applied Algebra, Algebraic Algor.\ and Error-Corr.\ Codes (Puerto
  Rico)}, volume 673 of {\em Lect. Notes in Comp. Science}, pages 89--104.
  Springer-Verlag, 1993.

\bibitem{CaEm00}
J.F. Canny and I.Z. Emiris.
\newblock A subdivision-based algorithm for the sparse resultant.
\newblock {\em J. ACM}, 47(3):417--451, May 2000.

\bibitem{CLO92}
D.~Cox, J.~Little, and D.~O'Shea.
\newblock {\em Ideals, Varieties, and Algorithms}.
\newblock Undergraduate Texts in Mathematics. Springer-Verlag, New York, 2nd
  edition, 1997.

\bibitem{CLO97}
D.~Cox, J.~Little, and D.~O'Shea.
\newblock {\em Using Algebraic Geometry}.
\newblock Springer-Verlag, New York, 1997.

\bibitem{Dandrea02}
C.~D'Andrea.
\newblock Macaulay style formulas for sparse resultants.
\newblock {\em Trans. Amer. Math. Soc.}, 354:2595--2629, 2002.

\bibitem{Eis94}
D.~Eisenbud.
\newblock {\em {C}ommutative {A}lgebra with a view toward {A}lgebraic
  {G}eometry}, volume 150 of {\em Graduate Texts in Math.}
\newblock Berlin, Springer-Verlag, 1994.

\bibitem{EM07}
M.~Elkadi and B.~Mourrain.
\newblock {\em Introduction {\`a} la r\'esolution des syst\`emes polynomiaux},
  volume~59 of {\em Math\'ematiques et Applications}.
\newblock Springer, 2007.

\bibitem{Emiris:1994:MBP:190347.190374}
I.Z. Emiris and A.~Rege.
\newblock Monomial bases and polynomial system solving.
\newblock In {\em Proceedings of the International Symposium on Symbolic and
  Algebraic Computation}, ISSAC '94, pages 114--122, New York, NY, USA, 1994.
  ACM.

\bibitem{GKZ94}
I.M. Gelfand, M.M. Kapranov, and A.V. Zelevinsky.
\newblock {\em Discriminants, {R}esultants and {M}ultidimensional
  {D}eterminants}.
\newblock Boston, Birkh{\"{a}}user, 1994.

\bibitem{Ka11}
S.~Kaspar.
\newblock Computing border bases without using a term ordering.
\newblock {\em Beitr\"age zur Algebra und Geometrie / Contributions to Algebra
  and Geometry}, pages 1--13, 2011.

\bibitem{KK05}
A.~Kehrein and M.~Kreuzer.
\newblock Characterizations of border bases.
\newblock {\em J. Pure Appl. Algebra}, 196(2-3):251--270, 2005.

\bibitem{KK06}
A.~Kehrein and M.~Kreuzer.
\newblock Computing border bases.
\newblock {\em J. Pure Appl. Algebra}, 205(2):279--295, 2006.

\bibitem{KR05}
M.~Kreuzer and L.~Robbiano.
\newblock {\em Computational Commutative Algebra 2}.
\newblock Springer, Heidelberg, 2005.

\bibitem{Mac02}
F.S. Macaulay.
\newblock Some formulae in elimination.
\newblock {\em Proc.\ London Math.\ Soc.}, 1(33):3--27, 1902.

\bibitem{BMnf99}
B.~Mourrain.
\newblock A new criterion for normal form algorithms.
\newblock In M.~Fossorier, H.~Imai, Shu Lin, and A.~Poli, editors, {\em Proc.
  AAECC}, volume 1719 of {\em LNCS}, pages 430--443. Springer, Berlin, 1999.

\bibitem{BM07}
B.~Mourrain.
\newblock {\em Symbolic-Numeric Computation}, chapter Pythagore's Dilemma,
  Symbolic-Numeric Computation, and the Border Basis Method, pages 223--243.
\newblock Trends in Mathematics. Birkh\"auser, 2007.

\bibitem{BMPhT00}
B.~Mourrain and Ph. Tr\'ebuchet.
\newblock Solving projective complete intersection faster.
\newblock In C.~Traverso, editor, {\em Proc.\ Intern.\ Symp.\ on Symbolic and
  Algebraic Computation}, pages 231--238. New-York, ACM Press., 2000.

\bibitem{BMPhT05}
B.~Mourrain and Ph. Tr\'ebuchet.
\newblock {Generalised normal forms and polynomial system solving}.
\newblock In M.~Kauers, editor, {\em {International Conference on Symbolic and
  Algebraic Computation}}, pages 253--260, Beijing, China, 2005. ACM New York,
  NY, USA.

\bibitem{MOURRAIN:2008:INRIA-00343103:1}
B.~Mourrain and Ph. Tr{\'e}buchet.
\newblock {Stable normal forms for polynomial system solving}.
\newblock {\em Theoretical Computer Science}, 409(2):229--240, 2008.

\bibitem{BMPhT12}
B.~Mourrain and Ph. Tr{\'e}buchet.
\newblock {Border basis representation of a general quotient algebra}.
\newblock In {\em {International Conference on Symbolic and Algebraic
  Computation (ISSAC)}}, pages 265--272, Grenoble, France, July 2012. ACM
  Press.

\bibitem{St:Elim}
B.\ Sturmfels.
\newblock Sparse elimination theory.
\newblock In D.~Eisenbud and L.~Robbiano, editors, {\em Proc. Computat.
  Algebraic Geom. and Commut. Algebra 1991}, pages 264--298, Cortona, Italy,
  1993. Cambridge Univ. Press.

\end{thebibliography}

\end{document}